\documentclass[a4paper,11pt,leqno,]{article}
\usepackage{amssymb, amsmath, amsthm, graphicx, txfonts, color}

\newtheorem{thm}{Theorem}[section]
\newtheorem{lem}[thm]{Lemma}

\newtheorem{rem}{Remark}[section]

\numberwithin{equation}{section}

\renewcommand{\a}{\alpha}
\renewcommand{\b}{\beta}
\newcommand{\e}{\varepsilon}
\newcommand{\de}{\delta}

\newcommand{\ga}{\gamma}
\renewcommand{\k}{\kappa}

\newcommand{\si}{\sigma}

\newcommand{\om}{\omega}

\newcommand{\Ga}{\Gamma}

\newcommand{\Om}{\Omega}
\newcommand{\lan}{\langle}
\newcommand{\ran}{\rangle}
\newcommand{\na}{\nabla}

\newcommand{\WTX}{\widetilde{X}}

\newcommand{\dist}{\operatorname{dist}\,}

\def\R{{\mathbb{R}}}

\def\Z{{\mathbb{Z}}}

\allowdisplaybreaks

\def\CONST3{{C_i}}
\def\CONSTa{{C_1}}
\def\CONST\{100\}{{C_1}}
\def\CONSTb{{C_2}}
\def\CONSTc{{C_3}}
\def\CONSTd{{C_4}}
\def\CONSTe{{C_5}}
\def\CONSTf{{C_6}}
\def\CONSTg{{C_7}}
\def\CONSTh{{C_8}}
\def\CONSTi{{C_9}}
\def\CONSTj{{C_{10}}}
\def\CONSTk{{C_{11}}}
\def\CONSTl{{C_{12}}}
\def\CONSTm{{C_{13}}}
\def\CONSTn{{C_{14}}}
\def\CONSTo{{C_{15}}}
\def\CONSTp{{C_{16}}}
\def\CONSTq{{C_{17}}}
\def\CONSTr{{C_{18}}}
\def\CONSTs{{C_{19}}}
\def\CONSTt{{C_{20}}}
\def\CONSTu{{C_{21}}}
\def\CONSTv{{C_{22}}}
\def\CONSTw{{C_{23}}}
\def\CONSTx{{C_{24}}}
\def\CONSTy{{C_{25}}}
\def\CONSTz{{C_{26}}}
\def\CONSTaa{{C_{27}}}
\def\CONSTab{{C_{28}}}
\def\CONSTac{{C_{29}}}
\def\CONSTad{{C_{30}}}
\def\CONSTae{{C_{31}}}
\def\CONSTaf{{C_{32}}}

\renewcommand{\theenumi}{\arabic{enumi}}
\renewcommand{\labelenumi}{(\theenumi)}
\renewcommand{\theenumii}{\alph{enumii}}
\renewcommand{\labelenumii}{(\theenumii)}

\def\thefootnote{{}}

\title{A stochastically perturbed mean curvature flow by colored noise}
\author{Satoshi Yokoyama$^{1)}$}
\date{\today}

\begin{document}
\maketitle

\begin{abstract}
We study the motion of the hypersurface $(\ga_t)_{t\ge0}$ 
evolving according to the mean curvature perturbed by $\dot{w}^Q$, 
the formal time derivative of the $Q$-Wiener process ${w}^Q$, in a two dimensional bounded domain.
Namely, 
we consider the equation describing the evolution of $\ga_t$ as a stochastic partial differential equation (SPDE)
with a multiplicative noise in the Stratonovich sense, whose inward velocity $V$ is determined by $V=\k\,+\,G \circ \dot{w}^Q$,
where $\k$ is the mean curvature and $G$ is a function 
determined from $\ga_t$. 
Already known results in which the noise depends on only time variable is not applicable to our equation.
To construct a local solution of the equation describing $\ga_t$,
we will derive a certain second order quasilinear SPDE
with respect to the signed distance function determined from $\ga_0$.
Then we construct the local solution making use of probabilistic tools and the classical Banach fixed-point theorem 
on suitable Sobolev spaces.
\end{abstract}

\footnote{
\hskip -6mm
$^{1)}$Department of Mathematics, School of Fundamental Science and Engineering,
Waseda University, 3-4-1 Okubo, Shinjuku-ku, Tokyo 169-8555, Japan\\
e-mail: satoshiyokoyama@aoni.waseda.jp, satoshi2@ms.u-tokyo.ac.jp}
\footnote{
\hskip -6mm
\textit{Abbreviated title $($running head$)$: stochastically perturbed mean curvature flow}}
\footnote{
\hskip -6mm
Author$^{1)}$ is supported in part by the JSPS KIKIN Grant 18K13430.}

\section{Introduction}

In this paper, we study the stochastic mean curvature flow affected by the $Q$-Wiener process moving on a two dimensional bounded domain.
So far, many results by several authors about the mean curvature flow have been found 
not only in the deterministic case but in the stochastic case.
As a typical example, the mean curvature flow appear in the moving of the {\it interface},
that is, the hypersurface which separates two different substantials, such as 
water and oil, or water and ice and so on.
The time evolution of an interface in a bounded domain are usually observed as a hypersurface which moves 
according to a law arising from a certain physical mutual operations, such as temperature, surface tension, pressure, and so on.
If the evolving velocity of each point of the interface is governed by the mean curvature of it, 
this flow is called the mean curvature flow.

As for the deterministic case, 
Evans and Spruck \cite{ES-2} can be raised up as a successful result,
that is, 
they construct the local solution of the solution which is governed by the mean curvature motion $V=\k$ in a bounded domain in $\R^d$
by means of replacing the equation by that determined by the level set of the signed distance determined uniquely from the interface $\ga_t$
on a tubular neighborhood. 
This method can be reduced to that of the second order nonlinear parabolic differential equation, for instance, see Lunardi \cite{Lu} for more details.
For further summarizing discussion and details for similar models using the level set, see also Bellettini \cite{B}.
%
As for the cases where the evolving velocity relies on 
both the mean curvature $\k$ and an extra force,
in early 1990s Chen and Reitich \cite{CR} consider the Stephan problem
in two phases in a bounded domain $\R^n$. In the proof of \cite{CR}, they derive the local existence and uniqueness for the solution
of the interface. Their method is as follows:
First they assume that the initial hypersurface $\ga_0$ is suitably smooth in such a way that the later discussion works well. 
Denoting by $X_0(x)$, $x\in M$ the parametrization of $\ga_0$, where $M$ is a suitable compact set, 
they study the equation of $u=u(x)$, where $u(x)$ denotes the {\it distance} from a point $X_0(x)$ of $\ga_0$,
in place of using the signed distance from $\ga_t$.
Indeed, a quasilinear second order partial differential equation of $u$ is obtained
and they construct its local solution by employing the Banach fixed point argument.
As for the result with some conservation law, Huisken \cite{H}, Elliott and Garket \cite{EG} study the conserved mean curvature flow and 
show the local existence and uniqueness of the solution.

In the stochastic case, Funaki \cite{F99} studies the mean curvature flow perturbed by white noise in time in a two dimensional domain
and he proves the local existence and uniqueness for such dynamics under the restriction
where the initial hypersurface $\ga_0$ is convex and $\ga_t$ stays convex.
Later, Weber \cite{W} extends the results of \cite{F99} to those in a bounded domain in $\R^n$ and removes the assumption of the convexity.
However, in \cite{W} the noise is white in time 
and their results are limited within in the case where the noise is independent of the spatial variable.
Lions and Souganidis's models \cite{LS1} and \cite{LS2} are similar to ours.
In addition, Souganidis and Yip \cite{SY} study the case where 
the noise is $\e\dot{W}$ with a small $\e>0$, where $\dot{W}$ is white in time and constant in space,
in a two dimensional setting. 
However, their methods and equations are different from ours.
We believe that it is valuable to try to argue using the Sobolev spaces 
as a different method to construct the unique solution of our equation.

The sharp interface limit of Allen Cahn equation for the mass conserving mean curvature motion is discussed by 
Chen, Hilhorst and Logak \cite{CHL}.
On the other hand, as for the stochastic case for the mass conserving model (stochastic Allen Cahn equation),
Funaki and Yokoyama \cite{FY} recently extend the deterministic case to the stochastic one perturbed by the noise
depending on only time and white in time.
Indeed, they prove that the solution of the dynamics exists uniquely up to the time as long as the hypersurface $\ga_t$ keeps being convex in a two dimensional bounded domain
and it can be derived from the sharp interface limit from the stochastic Allen Cahn equation.
In \cite{FY}, to avoid several technical difficulties, they consider under the assumption,
however, we believe and hope the assumption to be removed in not so far future,
In one dimensional case, Funaki \cite{F95} studies the sharp interface limit for the Allen Cahn equation perturbed by space-time white noise.
In higher spatial dimension $n\ge2$, the stochastic Allen Cahn equation with space-time white noise is no longer wellposed.
As for the sharp interface limit of the Allen Cahn equation with even $Q$-Wiener process, no clear answer has not been given 
and it is still an open problem.

Our interest is to study more natural situation.  Indeed, the existing results known so far are mainly in the deterministic case,
which implies that they are treated to be very ideal model which is governed by nonrandom dynamical systems.
However, seeing from the physical view point, it is not strange to take the unexpected factor such as noises into consideration.
%
%

The random forces appearing in models in the field such as physics or biology should depend on both time and spatial variables.
In this paper, we focus on the construction of the solution of the stochastic mean curvature motion
perturbed by the random force which contains the colored noise $\dot{W}^Q$
by making use of the methods of \cite{CR} and \cite{Lu}.

In Denis, Funaki and Yokoyama \cite{DFY}, they discuss the Wong-Zakai type limit to the equation describing a hypersurface which stays convex 
in a two dimensional bounded domain with a noise term which affects only to the normal direction on the hypersurface.
As for the equation which is obtained by the Wong-Zakai type limit in \cite{DFY},
under the assumption 
as long as the hypersurfaces starting from a convex initial hypersurface are still convex,
the dynamics of such hypersurface is represented by the solution of
a quasilinear parabolic second order SPDE of the mean curvature $\k$.
This can be rewritten into an SPDE whose principle part is of the divergence form
with regard to a certain quantity $u$, which is determined by $\k$, and
%
the local existence and uniqueness of $u$ 
is established by making use of the methods done by Debussche, de Moor and Hofmanov\`a \cite{DMH}.

%
Now we are going to state known results and methods which are related and likely to be thought of the nice candidates for our problem to be solved.
First, we want to emphasize that the idea used by Dirr, Luckhaus and Novaga \cite{DLN} is not applicable.
Indeed, in \cite{DLN}, the random force is only assumed to be additive noise with white in time and hence
they reduce the original SPDE to the PDE without the stochastic term, that is, {\it random} PDE which holds $\om$-wisely.

Furthermore, the methods by Weber and Otto \cite{WO} and 
\cite{DMH} 
are not applicable for our settings.
They treat a quasilinear parabolic SPDE which is of divergence form
with a time and spatial dependent noise 
but they do not cover our case. 
In our equation, the coefficient in front of the second order differential consists of $x$, $u$ and the gradient of $u$.
Due to such characteristics of the form of the coefficient,
our equation cannot be reduced to that whose principal part is of the divergence form. 
As a result, the energy estimates used in \cite{WO} and \cite{DMH} do not work well in our model.
In those two papers, the energy estimate in a suitable topology can be obtained thanks to nice assumption.
Since a nice energy estimate with a suitable topologies 
 cannot be obtained in our case,
in this paper we employ the Banach fixed point argument in a suitable Banach space as a different way.

We also need to mention the model discussed by Hofmanov\`a, R$\ddot{\text{o}}$ger and von Renesse \cite{HMvR}.
This model is apparently similar to but different from our model, 
in fact, they consider the graph of a stochastic mean curvature motion on a two dimensional torus, however, 
their equation does not coincide with ours. In fact, in our case the graph is not discussed.
In Denis, Matoussi and Stoica \cite{DMS}, the multiplicative noise containing both $u$ and $u'$ is discussed.
This assumption is stronger than ours. However, only the case of the Laplacian is discussed in \cite{DMS}, hence
the case of our nonlinear term is not covered.

Let us summarize the paper.
In Section 2, we derive our equation and then in Section 3, we formulate the problem.
In Section 4, we prove our main theorem.

\section{Motion of the interface}
In this section, we follow the same notations as used in \cite{CR}. Let $D$ be a bounded domain in $\R^n$. 
We denote by $\ga_0$ a hypersurface without boundary which is of $C^{\infty}$-class
satisfying $\dist(\ga_0,\partial D)>0$, where $\dist(A,B)=\inf_{x\in A,\,y\in B}|x-y|$ for $A,B\subset\R^n$.
Let $M \subset D$ be an $n-1$-dimensional $C^\infty$ manifold without boundary which is isometric to $\ga_0$ and 
\begin{align}
X_0(s'):\, M\rightarrow \ga_0
\end{align}
be a $C^{\infty}$ diffeomorphims. 
We denote by $\bold{n}(s')$ the outward normal vector to $\ga_0$ at $y=X_0(s')$ and define
a map $X$ from $M\times[-L,L]$ to $\R^n$ by
\begin{align}
X(s',s_n)= X_0(s')+s_n\bold{n}(s'),
\end{align}
where $L>0$ is sufficiently small in such a way that $X$ is diffeomorphims onto a neighborhood of $\ga_0$.
If necessary, transferring  the local charts of $M$, without loss of generality, we denote by $s'\in\R$ a generic point in $M$.
In addition, we denote by $\partial_{s}^i=\frac{\partial^i}{\partial s^i}$, $i=1,2$ the partial derivative in the local coordinates.
Therefore, we can take $M$ as a chart and fix it in what follows and
without loss of generality, we can discuss the dynamics of the hypersurface $\ga_t$ which will be defined later 
by using such $M$.
Here, $s_n$ stands for the signed distance from $y$ to $\ga_0$.
Let $S(y)=(S^1(y),\ldots,S^n(y))$ be the inverse function of $y=X(s)$ for $s=(s',s_n)$, where $s'\in M$.
In what follows, we consider the case $n=2$ due to some technical reason which will be stated later.
Note that $\na(S^1(y))\ne\bold{0}$ holds for every $y$.

Our goal is to construct a family of hypersurfaces $\{\ga_t\}_{0\le t\le T}$ defined on a probability space $(\Om,\mathcal{F},P)$ for a sufficiently small $T(\om)>0$
which has the form 
\begin{align}
\ga_t =\{ X(x,s_2)\,\Bigl|\, s_2=u(t,x,\om),\,x\in M,\, \om\in\Om\},\, 0\le t \le T,
\end{align}
where $u:\,[0,T]\times M \times \Om\rightarrow [-L,L]$ is a function  which is of $C([0,T];W^{2+2\a}_p(M))$, $\a\in(0,\frac12)$, almost surely,
and with $u(0,x)=0$ for every $x\in M$,  namely,
$\ga_t$ is the zero level set of the function 
\begin{align*}
  \Phi(t,y)=S^2(y)-u(t,S^{1}(y)),
\end{align*}
and the time evolution of $\ga_t$ is 
governed by
\begin{equation}  \label{eq:Modified EQ}
V= \k + G \circ \dot{w}^Q,
\end{equation}
where the notation $\circ$ means the Stratonovich sense, $V=V(t,x)$ is the inward normal velocity at $X(x,u(t,x))$ for $x\in M$ and
$\k=\k(t,x)$ is the mean curvature at $X(x,u(t,x))$ which is given by
\begin{align}\label{eq:V and k}
  V(t,x)=&-\partial_tX(t,u(t,x))\cdot\bold{n}(x)=\frac{\partial_t\Phi}{|\na_y\Phi|}\Bigr|_{y=X(x,u(t,x))}=-\frac{\partial_tu(t,x)}{|\na_y\Phi|}\Bigr|_{y=X(x,u(t,x))},\\
  \k(t,x)=&\na_y\bigl(\frac{\na_y\Phi}{|\na_y\Phi|}\bigr)\Bigr|_{y=X(x,u(t,x))}\notag\\
         =&-\frac{1}{|\na_y\Phi|}\Bigl(a(x,u(t,x),u'(t,x))u''(t,x)+b(x,u(t,x),u'(t,x))\Bigr)\Bigr|_{y=X(x,u(t,x))},\nonumber
\end{align}
where $a$ and $b$ are the functions as defined in \cite{CR},
namely,
\begin{align}\label{eq:a and b}
a(x,u,p)=&\frac{|\na S^1(X(x,u))|^2}{1+p^2|\na S^1(X(x,u))|^2},\\
b(x,u,p)=&p \text{Tr}(D^2S^1(X(x,u)))-\text{Tr}(D^2S^2(X(x,u)))\nonumber\\
&-\frac{p^2\,\lan D^2S^2(X(x,u))\na S^1(X(x,u)),\na S^1(X(x,u))\ran}{1+p^2|\na S^1(X(x,u))|^2}\nonumber\\
&-\frac{p^3\,\lan D^2S^1(X(x,u))\na S^1(X(x,u)),\na S^1(X(x,u))\ran}{1+p^2|\na S^1(X(x,u))|^2},\nonumber
\end{align}
where $\text{Tr}A$ denotes the trace of the matrix $A$, $D^2A$ means the Hessian matrix of $A$ 
and $\lan\cdot,\cdot\ran$ stands for the inner product of $\R^2$,
while $G$ is a function given by
\begin{align}\label{eq:G}
G=\frac{g(x,u(t,x))}{c(x,u(t,x),u'(t,x))},
\end{align}
where 
$g=g(x,u)$ is a function which is of a class $C^{\infty}(M\times [-L,L])$ and
$c$ is a function as defined in \cite{CR}, that is,
\begin{align}\label{eq:c}
c(x,u,p)=&|\na_y\Phi|\Bigl|_{y=X(x,u)}=\Bigl(1+p^2|\na S^1(X(x,u))|^2\Bigr)^{\frac12}.
\end{align}

 In \cite{B}, in order to construct a solution of \eqref{eq:Modified EQ} with $G\equiv1$ and with $w^Q$ replaced by some smooth deterministic function of $(t,x)$,
they consider the corresponding nonlinear second order parabolic differential equation
and show that its local solution uniquely exists.

 In the physical viewpoint, $G$ in \eqref{eq:G} stands for a quantity determined from $\ga_t$.
 We want to remark that the function $g$ in \eqref{eq:G} does not contain the derivative of $u$ and we restrict ourself to the case of $n=2$. 
 These restrictions are needed to construct the solution $u=u(t,x)$ of SPDE which will be stated in Section 3.
 For more precisely, in the proof of its uniqueness and existence of $u$, 
 several estimates for $u$ in a suitable topology are needed, and to obtain those estimates, 
 we need the estimates of the fundamental solution of the parabolic second order differential equation which 
 will appear in the sections below.
 In the procedure of the proof using those estimates, the assumptions as stated above are actually essential.
\section{Formulation of our problem}
%
In what follows, to simplify notations and make our discussion more understandable,
we restrict $w^Q$ to be in the form of $w^Q(t,x)=\psi(X(x,u(t,x)))B(t)$ for 
$\psi \in C^{\infty}_0(D)$ and $B$ is a standard Brownian motion defined on $(\Om,\mathcal{F},P)$.
\begin{rem}
Indeed, even if the noise of our equation is assumed to be as above, 
it is not too hard to extend the noises to more general one such as a sum of infinite number of independent Brownian motions.
This will be easily understandable by seeing the proof of our main theorem stated in the sections below.
\end{rem}

From \eqref{eq:Modified EQ}, \eqref{eq:V and k}, \eqref{eq:a and b}, \eqref{eq:G} and  \eqref{eq:c},
we get the following quasilinear stochastic differential equation (SPDE):
\begin{align}\label{eq:Main 1}
 \frac{\partial u}{\partial t}(t,x)&=a(x,u(t,x),u'(t,x))u''(t,x)+b(x,u(t,x),u'(t,x))\\
 &-g(x,u(t,x))  \circ \dot{w}^Q(t,x), \quad \nonumber\\
 u(0,x)&=0.\nonumber
\end{align}
Set 
\begin{align}\label{eq:Main 2}
 &f_1(x,u(t,x),u'(t,x))=a(x,u(t,x),u'(t,x)),\\
 &f_2(x,u(t,x),u'(t,x))=b(x,u(t,x),u'(t,x)),\nonumber\\
 &f_3(x,u(t,x))=-g(x,u(t,x)).\nonumber
\end{align}
Rewriting the Stratonovich integral of \eqref{eq:Main 1} into the Ito's form, we get
\begin{align}\label{eq:Main 1 Ito}
 \frac{\partial u}{\partial t}(t,x)=&\tilde{F}(x,u(t,x),u'(t,x),u''(t,x))
 +f_3(x,u(t,x)){\psi}(X(x,u(t,x)))\dot{B}_t,\\
 u(0,x)=&0,\notag
\end{align}
where $\tilde{F}(x,u,p,q)$ for $(x,u,p,q)\in M\times\R\times\R\times\R$ is given by
\begin{align}
\tilde{F}(x,u(t,x),u'(t,x),u''(t,x))=&f_1(x,u(t,x),u'(t,x))u''(t,x)\\
&+F_2(x,u(t,x),u'(t,x)),\notag
\end{align}
and
\begin{align}\label{eq:F_1 and F_2}
 F_2(x,u,p)=&f_2(x,u,p)+\frac12 f_3(x,u) \partial_uf_3(x,u) {\psi}^2(X(x,u))\\
 &+\frac12f_3^2(x,u) {\psi}(X(x,u))\na{\psi}(X(x,u))\cdot\bold{n}(x).\nonumber
\end{align}
The uniqueness and existence of the local solution in the deterministic case ($f_3\equiv0$) are discussed in \cite{Lu}. 
%

For an integer $K$, $\eta_K^{(1)}=\eta_K^{(1)}(u)\in[0,1]$ denotes a $C^\infty$-function for $u\in\R$
which is equal to $1$ for $|u|\le L(1-\frac{1}{1+K})$, while it is equal to $0$ for  $|u|\ge L(1-\frac{1}{2(1+K)})$.
Furthermore, let $\eta_K^{(2)}=\eta_K^{(2)}(u)$ be a $C^\infty$-function taking value in $[0,1]$ defined by
$1$ for $|u|\in [K^{-1},K]$, while $0$ for $|u|\le (2K)^{-1}$ or $|u'|\ge (2K)$.
Similarly, 
$\eta_K^{(3)}=\eta_K^{(3)}(u)\in[0,1]$ denotes a $C^\infty$-function for $u\in\R$
which is equal to $1$ for $|u|\le K$, while it is equal to $0$ for  $|u|\ge 2K$ and furthermore, 
$|\frac{d}{du}\eta_K^{(3)}(u)|\le\frac{2}{K}$.
Let us consider 
\begin{align}\label{eq:Mod Main 1 Ito}
 &\frac{\partial u}{\partial t}(t,x)=\tilde{F}_K(x,u(t,x),u'(t,x),u''(t,x))
 +f_{3,K}(x,u(t,x)) \psi(X(x,u(t,x)))\dot{B}_t,\\
 &u(0,x)=0,\notag
\end{align}
where
\begin{align*}
 \tilde{F}_{K}(x,u(t,x),u'(t,x),u''(t,x))=&\eta_K^{(1)}(u(t,x))\eta_K^{(2)}( \na S^1( X(x,u(t,x)) ) )\eta_K^{(3)}(u''(t,x))\\
 &\times\tilde{F}(x,u(t,x),u'(t,x),u''(t,x)),\notag\\
 f_{3,K}(x,u(t,x))=&\eta_K^{(1)}(u(t,x))f_3(x,u(t,x)).\nonumber
\end{align*}
%
Let us define
\begin{align}
\si_K^1=&\inf\Bigl\{t>0,\,\Bigl|\,\inf_{x\in M}|u(t,x)|\ge L(1-\frac{1}{1+K})\Bigr\},\notag\\
\si_K^2=&\inf\Bigl\{t>0,\,\Bigl|\,\inf_{x\in M}|\na S^1( X(x,u(t,x)) )|\le K^{-1}\Bigr\}\wedge
\inf\Bigl\{t>0,\,\Bigl|\,\sup_{x\in M}|\na S^1( X(x,u(t,x)) )|\ge 2K\Bigr\},\notag\notag\\
\si_K^3=&\inf\Bigl\{t>0,\,\Bigl|\,\sup_{x\in M}|u''(t,x)|\ge K\Bigr\}.\notag
\end{align}
Following the convention, we set $\si_K^i=\infty$, $i=1,2,3$ if the above set $\{\cdot\}$ is empty.
Let $\si_K=\min{(\si_K^1,\si_K^2,\si_K^3)}$.
Then, the solution $u^K=u^K(t,x)$ of \eqref{eq:Mod Main 1 Ito} 
coincides with the solution $u$ of \eqref{eq:Main 1 Ito}
on $0\le t\le \si_K$.
We wish construct a solution of \eqref{eq:Mod Main 1 Ito} later.
In what follows, we simply write $\tilde{F}_{K}$, $f_{3,K}$ for $\tilde{F}$, $f_{3}$, respectively.
Define the second order linear operator $A$ to be
\begin{align}\label{eq:operator A}
 Au=&\partial_q \tilde{F}(x,0,0,0)u''+\partial_p \tilde{F}(x,0,0,0)u'+\partial_u \tilde{F}(x,0,0,0)u,
\end{align}
(in our case, $u_0(x)=u_0'(x)=u_0''(x) \equiv 0$).
Note that the derivative of $\tilde{F}$ at $(x,0,0,0)$ in \eqref{eq:operator A} is characterized by $\ga_0$
and in addition for the initial hypersurface, 
we assume that $\ga_0$ satisfy $\inf_{x\in M}|\na S^1(X_0(x))|>0$.
Note that $\inf_{x\in M}\partial_q \tilde{F}(x,0,0,0) > 0$ holds by taking suitable large number $K$. 
%
Take such integer $K$ and fix it. 

Let $R>0$ and fix it. 
For $T>0$ which will be determined suitably later and $\a,\a_1>0$, we introduce the state space of $u$ as follows: 
\begin{align}\label{eq:space X}
 Y=&\Bigl\{ u \in L^p(\Omega;W^{0,2+2\a}_p([0,T]\times M) \cap W^{\a_1,2}_p([0,T]\times M)) \,\Bigl| \, u(0,x)\equiv0,\,x\in M,\,|||u||| \le R  \Bigr\},
\end{align}
where
\begin{align}\label{eq:distance of X}
 |||u|||=& \Bigl( E\Bigl[||u||_{ W^{0,2+2\a}_p([0,T]\times M) \cap W^{\a_1,2}_p([0,T]\times M)}^p \Bigr]\Bigr)^{\frac1p},
\end{align}
and
\begin{align}
  ||u||_{W^{0,2+2\a}_p([0,T]\times M) \cap W^{\a_1,2}_p([0,T]\times M)} = 
  ||u||_{W^{0,2+2\a}_p([0,T]\times M)} + ||u||_{W^{\a_1,2}_p([0,T]\times M)}.   \notag
\end{align}
In what follows, we set $\WTX=W^{0,2+2\a}_p([0,T]\times M) \cap W^{\a_1,2}_p([0,T]\times M)$.
We state our main theorem.
\begin{thm}
Let $\a_1\in(0,\frac{1-\de}{4})$ for arbitrary small $\de>0$ and $\a\in(0,\frac14)$ and furthermore,
let us assume $p>4$ satisfying $\a_1 p>1$ and $\a p > \frac12$. 
Let $T>0$ be a sufficiently small number satisfying \eqref{eq:the way of choosing T} below.
Then, the Equation \eqref{eq:Mod Main 1 Ito} has a unique solution $u(t,x)\in C([0,T];W^{2+2\a}_p(M))$, a.s., 
namely, \eqref{eq:Modified EQ} has a unique solution on $t\le\si_K\wedge T$ and $\si_K>0$ a.s.
\end{thm}

\section{Contraction principle}
%

\subsection{Setting for showing the contraction principle}
For given $u\in Y$, we define a map $\Ga$ from $Y$ to $Y$ by $v=\Ga(u)$ where $v$ is the solution of 
\begin{align}\label{eq:sol w}
 \frac{\partial v}{\partial t}&=Av(t,x)+\tilde{F}(x,u(t,x),u'(t,x),u''(t,x))\\
 &-Au(t,x)+f_3(x,u(t,x)){\psi}(X(x,u(t,x)))\dot{B}_t,\nonumber\\
 v(0,&x)=0.\nonumber 
\end{align}
%
In this case, the mild solution $v$ of \eqref{eq:sol w} is 
\begin{align}\label{eq:mild solution w}
 v(t,x)=&\int_0^t e^{(t-s)A}\tilde{F}(x,u(s,x),u'(s,x),u''(s,x))ds\\
 &+ \int_0^t e^{(t-s)A}f_3(x,u(s,x)){\psi}(X(x,u(s,x)))dB_s\nonumber\\
        &-\int_0^t e^{(t-s)A} Au(s,x)ds. \nonumber
\end{align}
In the case of $u=\Ga(0)$, $u$ satisfies
\begin{align}\label{eq:Main Linearized}
 \frac{\partial u}{\partial t}(t,x)=&Au(t,x)+\tilde{F}(x,0,0,0)+f_3(x,0) {\psi}(X(x,0))\dot{B}_t,\\
 u(0,x)=&0. \notag
\end{align}
We need to show that $\Ga(Y)\subset Y$ and $\Ga$ is a contraction map.
Setting $w=\Ga(u_1)-\Ga(u_2)$ for $u_1,u_2\in Y$, $w$ satisfies
\begin{align}\label{eq:solution of the difference}
 \frac{\partial w}{\partial t}=&Aw(t,x)+\Bigl[\tilde{F}(x,u_1(t,x),u_1'(t,x),u_1''(t,x))-\tilde{F}(x,u_2(t,x),u_2'(t,x),u_2''(t,x))\Bigr]\\
 &- A(u_1(t,x)-u_2(t,x))\nonumber\\
 &+ \Bigl(f_3(x,u_1(t,x)){\psi}(X(x,u_1(t,x)))-f_3(x,u_2(t,x)){\psi}(X(x,u_2(t,x)))\Bigr)\dot{B}_t,\notag\\
 w(0,x)=&0.\nonumber
\end{align}
The mild solution of \eqref{eq:solution of the difference} is given by
\begin{align}\label{eq:mild sol of the difference}
 w(t,x)=&\int_0^t e^{(t-s)A}\psi(s,x)ds\\
        &+ \int_0^t e^{(t-s)A}
        \Bigl(f_3(x,u_1(s,x)){\psi}(X(x,u_1(s,x)))\notag\\
        &\quad \quad -f_3(x,u_2(s,x)){\psi}(X(x,u_2(s,x)))\Bigr)dB_s\nonumber\\
       =&\int_0^t e^{(t-s)A}\psi(s,x)ds
        + \int_0^t e^{(t-s)A}
        \Bigl( f_4(x,u_1(s,x))-f_4(x,u_2(s,x)) \Bigr)dB_s\nonumber\\
        \equiv& w^{(d)}(t,x)+w^{(st)}(t,x),\nonumber
\end{align}
where
\begin{align*}
 \psi(t,x)=&\tilde{F}(x,u_1(t,x),u_1'(t,x),u_1''(t,x))-\tilde{F}(x,u_2(t,x),u_2'(t,x),u_2''(t,x))\\
 &-A\bigl(u_1(t,x)-u_2(t,x)\bigr),
\end{align*}
and
\begin{align*}
f_4(x,u(t,x))=f_3(x,u(t,x)){\psi}(X(x,u(t,x)).
\end{align*}
%
%
For the derivative $D_x^{2} f_4(x,u(x))$, we have
\begin{align}\label{eq:6-0}
&D_x^{2} f_4(x,u_1(x)) - D_x^{2}f_4(x,u_2(x))\\
=&\int_0^1 D_u D_x^{2}f_4(x,\si'u_1(x)+(1-\si')u_2(x))d\si' \,(u_1(x)-u_2(x)).\notag
\end{align}
%
Then, Sobolev embedding theorem leads to $W^{\a_1,p}([0,T])\subset C^{\ga_1}([0,T])$ for $\ga_1\in(0,\a_1-\frac1p)$
and $W^{2+2\a,p}(M) \subset C^{\ga_2}(M)$ for $\ga_2\in(0,2\a-\frac1p)$. 
Thus, it follows from the definition of the space $Y$ that $u$, $u'$ and $u''$ belong to $C^{\ga_1}([0,T])\cap C^{\ga_2}(M)$ a.s. 
%
We may assume that $|D_u D_x^{2}f_4|$ is finite a.s. by restricting $t$ to $t\wedge \si_K$.
%
%
In what follows we write $w^{(d)}$ and $w^{(st)}$ for short.
\subsection{Time Regularity for $w^{(d)}$}
In this section we study the time regularity for $w^{(d)}$.
$w^{(d)}$ satisfies
\begin{align}\label{eq:sol of (d) - 001}
 \frac{\partial w^{(d)}}{\partial t}=&Aw^{(d)}(t,x)+\psi(t,x),\quad w^{(d)}(0,x)=0.
\end{align}
First, Solonnikov \cite{So}, Theorem 5.4 shows that 
for $p>1$, 
\begin{align}\label{eq:Schauder of sol of (d) - 000-0}
 ||w^{(d)}||_{W^{1,2}_p([0,T]\times M)} &\le \CONSTa ||\psi||_{L^p([0,T]\times M)},
\end{align}
for $\CONSTa>0$, where $W^{1,2}_p([0,T]\times M)$ is the usual parabolic Sobolev space.
In our case, in place of the above space, we consider $W^{\a,2}_p$, $\a\in(0,\frac12)$, that is, 
$W^{\a,2}_p$ is a family of functions $u$ of $L^p$ having a finite value of
\begin{align*}
 ||u||_{W^{\a,2}_p([0,T]\times M)} =& 
 ||u||_{W^{0,2}_p([0,T]\times M)} +\Bigl( \sum_{i=0}^2 \int_M dx \int_0^T \int_0^T  \frac{|D_x^i(u(t,x)-u(s,x))|^p}{|t-s|^{1+\a p}} dsdt\Bigr)^{\frac1p},
\end{align*}
First, note that
\begin{align}\label{eq:Schauder of sol of (d) - 000-0}
 ||w^{(d)}||_{W^{\a_1,2}_p([0,T]\times M)} &\le \CONSTa ||\psi||_{L^p([0,T]\times M)},
\end{align}
holds from \eqref{eq:Schauder of sol of (d) - 000-0}.
On the other hand, following the method of Krylov \cite{K1} for example, we also obtain  
\begin{align}\label{eq:Schauder of sol of (d) - 000-Kry}
 ||w^{(d)}||_{L^p([0,T];H^{2+2\a}_p(M))} &\le \CONSTb ||\psi||_{L^p([0,T]; H^{2\a}_p(M))},
\end{align}
for $\CONSTb>0$, 
where $H^{2+2\a}_p(M)$ denotes a Bessel potential space, (see \cite{K1} for details). 
As is discussed in remark 2.1 of \cite{DMH}, we see $H^{2+2\a+\e}_p(M)) \subset W^{2+2\a}_p(M) \subset H^{2+2\a-\e}_p(M)$.
Therefore, taking $\e\rightarrow 0$, we get
\begin{align}\label{eq:Schauder of sol of (d) - 000-Kry and Hoh}
 ||w^{(d)}||_{L^p([0,T];W^{2+2\a}_p(M))} &\le \CONSTb ||\psi||_{L^p([0,T]; W^{2\a}_p(M))}.
\end{align}
As a result, 
\begin{align}\label{eq:Schauder of sol of (d) - 000-000-001}
 ||w^{(d)}||_{\WTX} 
 &\le \CONSTc ||\psi||_{W^{0,2\a}_p([0,T]\times M)\cap W^{\a_1,0}_p([0,T]\times M)},
\end{align}
follows from \eqref{eq:Schauder of sol of (d) - 000-0} and \eqref{eq:Schauder of sol of (d) - 000-Kry and Hoh}
for some $C_3>0$.
%
%
%
In what follows, $C_i$, $i\ge4$ denote positive constants independent of $T$ but which can  depend on $K$.
Since
\begin{align}\label{eq:psi-001-001}
\psi(t,x)
=& \int_0^1 \Bigl(\tilde{F}_u(x,\xi_{\tilde{\si}}(t,x))-\tilde{F}_u(x,\bold{0})\Bigr)\bigl(u_1(t,x)-u_2(t,x)\bigr)d\tilde{\si}\\
 &+\int_0^1 \Bigl(\tilde{F}_p(x,\xi_{\tilde{\si}}(t,x))-\tilde{F}_p(x,\bold{0})\Bigr)\bigl(u_1'(t,x)-u_2'(t,x)\bigr)d\tilde{\si}\notag\\
 &+\int_0^1 \Bigl(\tilde{F}_q(x,\xi_{\tilde{\si}}(t,x))-\tilde{F}_q(x,\bold{0})\Bigr)\bigl(u_1''(t,x)-u_2''(t,x)\bigr)d\tilde{\si}\notag,
\end{align}
where
\begin{align*}
 \xi_{\tilde{\si}}(t,x) =& \tilde{\si}\bigl(u_1(t,x),u_1'(t,x),u_1''(t,x)\bigr)+(1-\tilde{\si})\bigl(u_2(t,x),u_2'(t,x),u_2''(t,x)\bigr),\\
 \bold{0}=&\bigl(u_0(0,x),u_0'(0,x),u_0''(0,x)\bigr)=(0,0,0),
\end{align*}
Using
\begin{align*}
\sup_{x \in M} \Bigl|\tilde{F}_i(x,\xi_{\tilde{\si}}(t,x))-\tilde{F}_i(x,\bold{0})\Bigr|&\le C_4 T^{\ga_1},
\end{align*}
for sufficiently short time $t\le T$ and $C_4>0$.
(Use $|u^{(i)}(t \wedge T ,x)|\le [u^{(i)}(\cdot,x)]_{C^{\ga_1}([0,T])}T^{\ga_1}$),
we get
\begin{align}\label{eq:Lp[0,T]times M}
 ||\psi||_{L^p([0,T]\times M)}^p\le& T^{p\ga_1} C_5 [ u_1-u_2]_{L^p}^p,
\end{align}
where
\begin{align*}
[ u_1-u_2]_{L^p} = \Bigl(\sum_{j=0}^2 ||D^{j}_x\bigl(u_1-u_2\bigr)||_{ L^p((0,T)\times M) }^p\Bigr)^{\frac{1}{p}}.
\end{align*}
Therefore,
\begin{align}\label{eq: Schauder time sobolev}
 ||\psi||_{L^p([0,T]\times M)}^p \le&  C_6  T^{p\ga_1} ||u_1-u_2||_{\WTX}^p,
\end{align}
holds. Next, note that
\begin{align}\label{eq:psi-002-001}
\psi(t,x)
=& \int_0^1 \Bigl(\tilde{F}_u(x,\xi_{\tilde{\si}}(t,x))-\tilde{F}_u(x,\bold{0})\Bigr)\bigl(u_1(t,x)-u_2(t,x)\bigr)d\tilde{\si}\\
 &+\int_0^1 \Bigl(\tilde{F}_p(x,\xi_{\tilde{\si}}(t,x))-\tilde{F}_p(x,\bold{0})\Bigr)\bigl(u_1'(t,x)-u_2'(t,x)\bigr)d\tilde{\si}\notag\\
 &+\int_0^1 \Bigl(\tilde{F}_q(x,\xi_{\tilde{\si}}(t,x))-\tilde{F}_q(x,\bold{0})\Bigr)\bigl(u_1''(t,x)-u_2''(t,x)\bigr)d\tilde{\si}\notag,
\end{align}
can be rewritten into
\begin{align*}
 & \int_0^1 \Bigl(\tilde{F}_u(x,\xi_{\tilde{\si}}(t,x))-\tilde{F}_u(x,\xi_{\tilde{\si}}(s,x))\Bigr)\bigl(u_1(t,x)-u_2(t,x)\bigr)d\tilde{\si}\\
 &+\int_0^1 \Bigl(\tilde{F}_u(x,\xi_{\tilde{\si}}(s,x))-\tilde{F}_u(x,\bold{0})\Bigr)\bigl(u_1(t,x)-u_2(t,x)\bigr)d\tilde{\si}\notag\\
 & \int_0^1 \Bigl(\tilde{F}_p(x,\xi_{\tilde{\si}}(t,x))-\tilde{F}_p(x,\xi_{\tilde{\si}}(s,x))\Bigr)\bigl(u_1'(t,x)-u_2'(t,x)\bigr)d\tilde{\si}\notag\\
 &+\int_0^1 \Bigl(\tilde{F}_p(x,\xi_{\tilde{\si}}(s,x))-\tilde{F}_p(x,\bold{0})\Bigr)\bigl(u_1'(t,x)-u_2'(t,x)\bigr)d\tilde{\si}\notag\\
 & \int_0^1 \Bigl(\tilde{F}_q(x,\xi_{\tilde{\si}}(t,x))-\tilde{F}_q(x,\xi_{\tilde{\si}}(s,x))\Bigr)\bigl(u_1''(t,x)-u_2''(t,x)\bigr)d\tilde{\si}\notag\\
 &+\int_0^1 \Bigl(\tilde{F}_q(x,\xi_{\tilde{\si}}(s,x))-\tilde{F}_q(x,\bold{0})\Bigr)\bigl(u_1''(t,x)-u_2''(t,x)\bigr)d\tilde{\si}\notag.
\end{align*}
Thus, 
\begin{align}\label{eq:time-det-001}
\psi(t,x)-\psi(s,x)
=& \int_0^1 \Bigl(\tilde{F}_u(x,\xi_{\tilde{\si}}(t,x))-\tilde{F}_u(x,\xi_{\tilde{\si}}(s,x))\Bigr)\bigl(u_1(t,x)-u_2(t,x)\bigr)d\tilde{\si}\\
 &+\int_0^1 \Bigl(\tilde{F}_u(x,\xi_{\tilde{\si}}(s,x))-\tilde{F}_u(x,\bold{0})\Bigr)\Bigl(\bigl(u_1(t,x)-u_2(t,x)\bigr) - \bigl(u_1(s,x)-u_2(s,x)\bigr)\Bigr)d\tilde{\si}\notag\\
 &+\int_0^1 \Bigl(\tilde{F}_p(x,\xi_{\tilde{\si}}(t,x))-\tilde{F}_p(x,\xi_{\tilde{\si}}(s,x))\Bigr)\bigl(u_1'(t,x)-u_2'(t,x)\bigr)d\tilde{\si}\notag\\
 &+\int_0^1 \Bigl(\tilde{F}_p(x,\xi_{\tilde{\si}}(s,x))-\tilde{F}_p(x,\bold{0})\Bigr)\Bigl(\bigl(u_1'(t,x)-u_2'(t,x)\bigr) - \bigl(u_1'(s,x)-u_2'(s,x)\bigr)\Bigr)d\tilde{\si}\notag\\
 &+\int_0^1 \Bigl(\tilde{F}_q(x,\xi_{\tilde{\si}}(t,x))-\tilde{F}_q(x,\xi_{\tilde{\si}}(s,x))\Bigr)\bigl(u_1''(t,x)-u_2''(t,x)\bigr)d\tilde{\si}\notag\\
 &+\int_0^1 \Bigl(\tilde{F}_q(x,\xi_{\tilde{\si}}(s,x))-\tilde{F}_q(x,\bold{0})\Bigr)\Bigl(\bigl(u_1''(t,x)-u_2''(t,x)\bigr) - \bigl(u_1''(s,x)-u_2''(s,x)\bigr)\Bigr)d\tilde{\si}\notag.
\end{align}
Then there exists a some constant $C_7$ such that 
\begin{align}\label{eq:time-det-002}
  \Bigl|\tilde{F}_k(x,\xi_{\tilde{\si}}(t,x))-\tilde{F}_k(y,\xi_{\tilde{\si}}(s,x))\Bigr|
    \le& C_7 \sum_{j=0}^2\Bigl(|D^j(u_1(t,x)-u_1(s,x))| +   |D^j(u_2(t,x)-u_2(s,x))|\Bigr),
\end{align}
holds for $k=u,p,q$.
%
Furthermore, since 
\begin{align*}
 \xi_{\tilde{\si}}(t,x) =& 
 \tilde{\si}\bigl(u_1(t,x)-u_1(0,x),u_1'(t,x)-u_1'(0,x),u_1''(t,x)-u_1''(0,x)\bigr)\\
 &+(1-\tilde{\si})\bigl(u_2(t,x)-u_2(0,x),u_2'(t,x)-u_2'(0,x),u_2''(t,x)-u_2''(0,x)\bigr),
\end{align*}
using the mean value theorem, we get
\begin{align}\label{eq:time-det-003}
\Bigl|\tilde{F}_k(x,\xi_{\tilde{\si}}(t,x))-\tilde{F}_k(x,\bold{0})\Bigr| 
\le& C_8 T^{\ga_1}\Bigl( \sum_{i=1}^2\sum_{j=0}^2\sup_{x\in M}[D^j u_i(x)]_{C^{\ga_1}([0,T])} \Bigr)\\
\le& C_8 T^{\ga_1} \Bigl(||u_1||_{\WTX} + ||u_2||_{\WTX}\Bigr), \notag
\end{align}
for $t\in[0,T]$ and $k=u,p,q$.
%
On the other hand, recalling again $D^j(u_i(0,x))=0$, $i=1,2$, $j=0,1,2$, we have
\begin{align}\label{eq:time-det-004}
 \sup_{x\in M}|D^j(u_1(t,x)-u_2(t,x))|=&\sup_{x\in M}|D^j(u_1(t,x)-u_2(t,x))-D^j(u_1(0,x)-u_2(0,x))|\\
 \le& [D^j(u_1-u_2)]_{C^{\ga_1,0}([0,T]\times M)}T^{\ga_1} \notag\\
 \le& ||u_1-u_2||_{\WTX}T^{\ga_1} \notag
 \end{align}
for $j=0,1,2$ and  $T>0$ and every $\ga_1\in(0,\a_1-\frac{1}{p})$ (using Sobolev embedding).
As a result, applying \eqref{eq:time-det-002} to \eqref{eq:time-det-004} for \eqref{eq:time-det-001},
we get
\begin{align}\label{eq:time-det-005}
\int_M dx \int_0^T\int_0^T 
\frac{|\psi(t,x)-\psi(s,x)|^p}{(t-s)^{1+ \a_1p}} dtds \le & C_9 T^{\ga_1 p} ||u_1 - u_2||_{\WTX}^p.
\end{align}
As a result, we obtain from  \eqref{eq:Lp[0,T]times M} and  \eqref{eq:time-det-005}
\begin{align}\label{eq:time-det-006}
 ||\psi||_{W^{\a_1,0}_p([0,T]\times M)}^p \le  C_{10} T^{\ga_1 p}||u_1-u_2||_{\WTX}^p.
\end{align}

%
\subsection{Space Regularity for $w^{(d)}$}
Similarly to the computation of the time regularity of $w^{(d)}$, 
$\psi(t,x)$ can be written into
\begin{align*}
 & \int_0^1 \Bigl(\tilde{F}_u(x,\xi_{\tilde{\si}}(t,x))-\tilde{F}_u(x,\bold{0})\Bigr)\bigl(u_1(t,x)-u_2(t,x)\bigr)d\tilde{\si}\\
 &-\int_0^1 \Bigl(\tilde{F}_u(y,\xi_{\tilde{\si}}(t,y))-\tilde{F}_u(y,\bold{0})\Bigr)\bigl(u_1(t,x)-u_2(t,x)\bigr)d\tilde{\si}\notag\\
 &+\int_0^1 \Bigl(\tilde{F}_u(y,\xi_{\tilde{\si}}(t,y))-\tilde{F}_u(y,\bold{0})\Bigr)\bigl(u_1(t,x)-u_2(t,x)\bigr)d\tilde{\si}\notag\\
 &+\int_0^1 \Bigl(\tilde{F}_p(x,\xi_{\tilde{\si}}(t,x))-\tilde{F}_p(x,\bold{0})\Bigr)\bigl(u_1'(t,x)-u_2'(t,x)\bigr)d\tilde{\si}\notag\\
 &-\int_0^1 \Bigl(\tilde{F}_p(y,\xi_{\tilde{\si}}(t,y))-\tilde{F}_p(y,\bold{0})\Bigr)\bigl(u_1'(t,x)-u_2'(t,x)\bigr)d\tilde{\si}\notag\\
 &+\int_0^1 \Bigl(\tilde{F}_p(y,\xi_{\tilde{\si}}(t,y))-\tilde{F}_p(y,\bold{0})\Bigr)\bigl(u_1'(t,x)-u_2'(t,x)\bigr)d\tilde{\si}\notag\\
 &+\int_0^1 \Bigl(\tilde{F}_q(x,\xi_{\tilde{\si}}(t,x))-\tilde{F}_q(x,\bold{0})\Bigr)\bigl(u_1''(t,x)-u_2''(t,x)\bigr)d\tilde{\si}\notag\\
 &-\int_0^1 \Bigl(\tilde{F}_q(y,\xi_{\tilde{\si}}(t,y))-\tilde{F}_q(y,\bold{0})\Bigr)\bigl(u_1''(t,x)-u_2''(t,x)\bigr)d\tilde{\si}\notag\\
 &+\int_0^1 \Bigl(\tilde{F}_q(y,\xi_{\tilde{\si}}(t,y))-\tilde{F}_q(y,\bold{0})\Bigr)\bigl(u_1''(t,x)-u_2''(t,x)\bigr)d\tilde{\si}\notag
\end{align*}
Thus, 
\begin{align}\label{eq:psi-002-002}
&\psi(t,x)-\psi(t,y)\\
=
 & \int_0^1 \Bigl(\tilde{F}_u(x,\xi_{\tilde{\si}}(t,x))-\tilde{F}_u(y,\xi_{\tilde{\si}}(t,y))\Bigr)\bigl(u_1(t,x)-u_2(t,x)\bigr)d\tilde{\si}\notag\\
 &+\int_0^1 \Bigl(\tilde{F}_u(y,\xi_{\tilde{\si}}(t,y))-\tilde{F}_u(y,\bold{0})\Bigr)\Bigl(\bigl(u_1(t,x)-u_2(t,x)\bigr)-\bigl(u_1(t,y)-u_2(t,y)\bigr)\Bigr)d\tilde{\si}\notag\\
 &+\int_0^1 \Bigl(\tilde{F}_u(y,\bold{0})-\tilde{F}_u(x,\bold{0})\Bigr)\bigl(u_1(t,x)-u_2(t,x)\bigr)d\tilde{\si}\notag\\
 &+\int_0^1 \Bigl(\tilde{F}_p(x,\xi_{\tilde{\si}}(t,x))-\tilde{F}_p(y,\xi_{\tilde{\si}}(t,y))\Bigr)\bigl(u_1'(t,x)-u_2'(t,x)\bigr)d\tilde{\si}\notag\\
 &+\int_0^1 \Bigl(\tilde{F}_p(y,\xi_{\tilde{\si}}(t,y))-\tilde{F}_p(y,\bold{0})\Bigr)\Bigl(\bigl(u_1'(t,x)-u_2'(t,x)\bigr)-\bigl(u_1'(t,y)-u_2'(t,y)\bigr)\Bigr)d\tilde{\si}\notag\\
 &+\int_0^1 \Bigl(\tilde{F}_p(y,\bold{0})-\tilde{F}_p(x,\bold{0})\Bigr)\bigl(u_1'(t,x)-u_2'(t,x)\bigr)d\tilde{\si}\notag\\
 &+\int_0^1 \Bigl(\tilde{F}_q(x,\xi_{\tilde{\si}}(t,x))-\tilde{F}_q(y,\xi_{\tilde{\si}}(t,y))\Bigr)\bigl(u_1''(t,x)-u_2''(t,x)\bigr)d\tilde{\si}\notag\\
 &+\int_0^1 \Bigl(\tilde{F}_q(y,\xi_{\tilde{\si}}(t,y))-\tilde{F}_q(y,\bold{0})\Bigr)\Bigl(\bigl(u_1''(t,x)-u_2''(t,x)\bigr)-\bigl(u_1''(t,y)-u_2''(t,y)\bigr)\Bigr)d\tilde{\si}\notag\\
 &+\int_0^1 \Bigl(\tilde{F}_q(y,\bold{0})-\tilde{F}_q(x,\bold{0})\Bigr)\bigl(u_1''(t,x)-u_2''(t,x)\bigr)d\tilde{\si}.\notag
\end{align}
By the mean value theorem, we get
\begin{align}\label{eq:psi-002-002-001}
  &\Bigl|\tilde{F}_k(x,\xi_{\tilde{\si}}(t,x))-\tilde{F}_k(y,\xi_{\tilde{\si}}(t,y))\Bigr|
    \le C_7 \sum_{j=0}^2\Bigl(|D^j(u_1(t,x)-u_1(t,y))| +   |D^j(u_2(t,x)-u_2(t,y))|\Bigr),
\end{align}
by retaking $C_7$ if necessary, (see \eqref{eq:time-det-002}).
In particular, clearly,
\begin{align}\label{eq:psi-002-002-001-001}
  &\Bigl| \tilde{F}_k(y,\bold{0})-\tilde{F}_k(x,\bold{0}) \Bigr| \le C_{11} |x-y|,
\end{align}
holds. 
Applying \eqref{eq:psi-002-002-001}, \eqref{eq:psi-002-002-001-001}, \eqref{eq:time-det-003} and \eqref{eq:time-det-004} for \eqref{eq:psi-002-002}, we get
\begin{align}\label{eq:psi-002-002-006}
&\int_0^T dt \int_M\int_M 
\frac{|\psi(t,x)-\psi(t,y)|^p}{(x-y)^{1+ 2\a p}} dxdy 
\le  C_{12} T^{\ga_1 p} ||u_1 - u_2||_{\WTX}^p.
\end{align}
%
As a result, we obtain from \eqref{eq:Lp[0,T]times M} and \eqref{eq:psi-002-002-006},
\begin{align}\label{eq:estimate of psi (d) - 001}
 ||\psi||_{W^{0,2\a}_p([0,T]\times M)}^p \le C_{13} T^{\ga_1 p}||u_1-u_2||_{\WTX}^p.
\end{align}
for some $C_{13}>0$.
%
%
\subsection{Estimates for $w^{(d)}$}
Now we are ready to estimate the time and spacial regularity for $w^{(d)}$.
From \eqref{eq:time-det-006} and \eqref{eq:estimate of psi (d) - 001}, we get
\begin{align}
 ||\psi||_{W^{\a_1,0}_p([0,T]\times M) \cap W^{0,2\a}_p([0,T]\times M)}^p \le C_{14} T^{\ga_1 p}||u_1-u_2||_{\WTX}^p.
\end{align}
Therefore, we get
\begin{align}
 ||w^{(d)}||_{\WTX}^p \le& C_{15} T^{p\ga_1} ||u_1-u_2||_{\WTX}^p.
\end{align}
Taking sufficiently small $T>0$, we get
\begin{align}\label{eq:deter fixed pt}
 \Bigl(E\Bigl[||w^{(d)}||_{\WTX}^p\Bigr]\Bigr)^{\frac1p} \le& K_d \Bigl(E\Bigl[||u_1-u_2||_{\WTX}^p\Bigr]\Bigr)^{\frac1p},
\end{align}
for $K_d=C_{15}^{\frac{1}{p}} T^{\ga_1} \le \frac12$. 
Clearly, this is possible.

\subsection{Time regularity for $w^{(st)}$}

Note that $w^{(st)}$ is rewritten as the mild solution using the fundamental solution $g(t,x,y)$, $t>0$, $x,y\in M$ of $\partial_t-A$, 
here recall that $A$ is the second order strong elliptic differential operator on $L^2(M)$ defined by \eqref{eq:operator A},
namely,
\begin{align}\label{eq:the stochastic integral w(st)}
 w^{(st)}(t,x)=&\int_0^t \int_M g(t-\si,x,z)
        \Bigl( f_4(z,u_1(\si,z))-f_4(z,u_2(\si,z))\Bigr)dzdB_\si,
\end{align}
holds.
We will check the time regularity of $w^{(st)}(t,x)$.
In this section. $C_p>0$ and $C_{p,K}>0$, which will appear later,
denote generic constants which might change its value line by line for simplicity of notations.
For $k=0,1,2$, using the estimate for the fundamental solution (see Appendix),
we obtain the estimate as follows:
\begin{align}\label{eq:6-0-0-0}
 &
 E
 \Bigl[\Bigl|
 D_x^k w^{(st)}(t,x)
 \Bigr|^p\Bigr]\\
 \le&
 C_p\Bigl(E
 \Bigl[
 \int_0^t \Bigl| 
 \int_M D_x^k g(t-\si,x,z)
 \Bigl( f_4(z,u_1(\si,z))-f_4(z,u_2(\si,z))\Bigr)dz\Bigr|^2 d\si
 \Bigr]\Bigr)^{\frac{p}{2}} \notag\\
 =&
 C_p\Bigl(E
 \Bigl[
 \int_0^t \Bigl| 
 \int_M  (D_z^{-1})^k D_x^k g(t-\si,x,z)  
 D_z^k\Bigl( f_4(z,u_1(\si,z))-f_4(z,u_2(\si,z))\Bigr)dz\Bigr|^2 d\si
 \Bigr]\Bigr)^{\frac{p}{2}} \notag\\
 \le&
 C_{p,K}\Bigl(
 \int_0^t  d\si 
 \int_M |  (D_z^{-1})^k D_x^k g(t-\si,x,z)|^2 dz\notag\\
 &\quad \times E\Bigl[\int_M |u_1(\si,z)-u_2(\si,z)|^2 dz
 \Bigr]\Bigr)^{\frac{p}{2}}  \notag\\
 \le&
 C_{p,K}  \Bigl(
 \Biggl(\int_0^t  
 \Bigl(\int_M |(D_z^{-1})^k D_x^k g(t-\si,x,z)|^2 dz\Bigr)^{\frac{p}{p-2}}
 d\si 
 \Biggr)^{\frac{p-2}{p}}\notag\\
 &\quad\times\Biggl(\int_0^t\Bigl(
 E\Bigl[\int_M |u_1(\si,z)-u_2(\si,z)|^2 dz\Bigr]\Bigr)^{\frac{p}{2}}d\si\Biggr)^{\frac{2}{p}}
 \Bigr)^{\frac{p}{2}} 
 \notag\\
 \le&
 C_{p,K}  \Bigl(\int_0^t (t-\si)^{-\frac12\frac{p}{p-2}}  d\si \Bigr)^{\frac{p-2}{2}}
 \Bigl(
 \Biggl(\int_0^t\Bigl(
 E\Bigl[ \int_M |u_1(\si,z)-u_2(\si,z)|^2 dz\Bigr]\Bigr)^{\frac{p}{2}}d\si\Biggr)^{\frac{2}{p}}
 \Bigr)^{\frac{p}{2}}
 \notag\\
 \le&
 C_{p,K} 
 t^{\bigl(1-\frac12\frac{p}{p-2}\bigr)\frac{p-2}{2}}
 \int_0^t\int_M E\Bigl[ |u_1(\si,z)-u_2(\si,z)|^p\Bigr] dzd\si,
 \notag
\end{align}
where $D_x^k=\frac{\partial^k}{\partial x^k}$ and $D_x^{-1}f$ denotes the primitive function of $f$
and $(D_x^{-1})^2f=D_x^{-1}(D_x^{-1}f)$,  
and at the second inequality we have used \eqref{eq:6-0}  
and 
\begin{align*}
\sup_{x\in M}\Bigl|\int_0^1 D_u D_x^{2}f_4(x,\si'u_1(x)+(1-\si')u_2(x))d\si'\Bigr|
\end{align*}
is bounded
and then used the fact that $-\frac12\frac{p}{p-2}>-1$ holds because of $p>4$.
%
From the estimate of \eqref{eq:6-0-0-0}, we have
\begin{align}\label{eq:6-0-0-1-0-1-001}
 E\Bigl[\Bigl|
 D_x^k w^{(st)}(t,x)
 \Bigr|^p\Bigr]
 \le &C_{p,K}
 t^{\bigl(1-\frac12\frac{p}{p-2}\bigr)\frac{p-2}{2}+1}
 \int_M E\Bigl[ ||u_1(\cdot,z)-u_2(\cdot,z)||^p_{C^{\ga_1}([0,T])} \Bigr] dz
\end{align}
for every $(t,x)\in [0,T]\times M$. 
Integrating the both hand side at the above inequality by $t\in[0,T]$ and $x \in M$, we get
\begin{align}\label{eq:6-0-0-1-0-1}
 \int_M E
 \Bigl[
 ||D_x^k w^{(st)}(x)||_{L^p([0,T])}^p\Bigr] dx
 \le&
 C_{p,K}
 T^{\frac14p}
 \int_M E\Bigl[ ||\bigl(u_1(\cdot,z)-u_2(\cdot,z)\bigr)||^p_{C^{\ga_1}([0,T])} \Bigr] dz.
\end{align}
%
%
Note that
\begin{align}\label{eq:diff of w(st)}
 &w^{(st)}(t,x)-w^{(st)}(s,x)\\
 =&\int_0^s \int_M \bigl(g(t-\si,x,z)-g(s-\si,x,z)\bigr)
        \Bigl( f_4(z,u_1(\si,z))-f_4(z,u_2(\si,z))\Bigr)dzdB_\si \notag\\
 &+\int_s^t \int_M g(t-\si,x,z)
        \Bigl( f_4(z,u_1(\si,z))-f_4(z,u_2(\si,z))\Bigr)dzdB_\si. \notag
\end{align}
For $k=0,1,2$, we have
\begin{align}\label{eq:7-0-0}
 &
 E
 \Bigl[\Bigl|
 D_x^k(w^{(st)}(t,x)-w^{(st)}(s,x)
 \Bigr|^p\Bigr]\\
 \le&
 C_p\Bigl(E
 \Bigl[
 \int_0^s \Bigl| 
 \int_M D_x^k \bigl(g(t-\si,x,z)-g(s-\si,x,z)\bigr) \notag\\
 &\quad \quad\times
        \Bigl( f_4(z,u_1(\si,z))-f_4(z,u_2(\si,z))\Bigr) dz\Bigr|^2 d\si
 \Bigr]\Bigr)^{\frac{p}{2}} \notag\\
 &+
 C_p\Bigl(E
 \Bigl[
 \int_s^t \Bigl| 
 \int_M D_x^k
 g(t-\si,x,z)\notag\\
 &\quad \quad\times 
        \Bigl( f_4(z,u_1(\si,z))-f_4(z,u_2(\si,z))\Bigr)
 dz\Bigr|^2 d\si
 \Bigr]\Bigr)^{\frac{p}{2}} \notag\\
 =&
 C_p\Bigl(E
 \Bigl[
 \int_0^s \Bigl| 
 \int_M    (D_z^{-1})^kD_x^k \bigl(g(t-\si,x,z)-g(s-\si,x,z)\bigr) \notag \\
 &\quad \quad \quad\times
         D_z^k \Bigl( f_4(z,u_1(\si,z))-f_4(z,u_2(\si,z))\Bigr) dz\Bigr|^2 d\si
 \Bigr]\Bigr)^{\frac{p}{2}} \notag\\
 &+
 C_p\Bigl(E
 \Bigl[
 \int_s^t \Bigl| 
 \int_M    (D_z^{-1})^kD_x^k
 g(t-\si,x,z)\notag \\
 &\quad \quad \quad\times
        D_z^k \Bigl( f_4(z,u_1(\si,z))-f_4(z,u_2(\si,z))\Bigr)
 dz\Bigr|^2 d\si
 \Bigr]\Bigr)^{\frac{p}{2}}\notag\\
 \equiv& III + IV.\notag
\end{align}
As for $IV$, we get
\begin{align}\label{eq:7-0-1}
 IV \le & 
 C_p\Bigl(E
 \Bigl[
 \int_s^t   d\si
 \int_M \Bigl|    (D_z^{-1})^kD_x^k g(t-\si,x,z) \Bigr|^2 dz \\
 &\quad \quad \quad\times
 \int_M \Bigl| D_z^k\Bigl( f_4(z,u_1(\si,z))-f_4(z,u_2(\si,z))\Bigr)\Bigr|^2
 dz
 \Bigr]\Bigr)^{\frac{p}{2}} \notag\\
 \le &
 C_{p,K} \Bigl(
 \int_s^t  (t-\si)^{-\frac12}d\si 
 E
 \Bigl[
 \int_M \Bigl| u_1(\si,z)-u_2(\si,z) \Bigr|^2
 dz
 \Bigr]\Bigr)^{\frac{p}{2}} \notag\\
 \le &
 C_{p,K} 
 \Bigl(
 \int_s^t  (t-\si)^{-\frac12}d\si 
 E
 \Bigl[
 \Bigl(\int_M \Bigl| u_1(\si,z)-u_2(\si,z) \Bigr|^p dz\Bigr)^{\frac{2}{p}}
 \Bigr]\Bigr)^{\frac{p}{2}} \notag\\
 \le &
 C_{p,K}
 \Bigl(
 \int_s^t  (t-\si)^{-\frac12}d\si 
 \Bigr)^{\frac{p}{2}}
 \Bigl(
 E
 \Bigl[
 \Bigl(\int_M || u_1(\cdot,z)-u_2(\cdot,z) ||_{C^{\ga_1}([0,T])}^p dz\Bigr)^{\frac{2}{p}}
 \Bigr]\Bigr)^{\frac{p}{2}} \notag \\
 \le &
 C_{p,K}
 (t-s)^{\frac{p}{4}}
 E
 \Bigl[
 \int_M || u_1(\cdot,z)-u_2(\cdot,z) ||_{C^{\ga_1}([0,T])}^p dz
 \Bigr] \notag
\end{align}
for some $\ga_1>0$, which is the same number taken as in the deterministic term, where we have used Schwarz inequality at the first inequality and 
\eqref{eq:6-0} at the second one, and then Holder's inequality at the third and the last inequality.
On the other hand, clearly
\begin{align*}
 III \le & 
 C_{p,K}\Biggl(E
 \Bigl[
 \int_0^s  d\si
 \int_M \Bigl|   (D_z^{-1})^k D_x^k \bigl(g(t-\si,x,z)-g(s-\si,x,z)\bigr) \Bigl|^2 dz \\
 &\quad \quad \quad\times
 \Bigl(\int_M \Bigl| D_z^2 \Bigl( f_4(z,u_1(\si,z))-f_4(z,u_2(\si,z))\Bigr) \Bigr|^p dz\Bigr)^{\frac{2}{p}} 
 \Bigr]\Biggr)^{\frac{p}{2}}. \notag 
\end{align*}
As for the second term of the right hand side, using \eqref{eq:6-0}, we have
\begin{align}\label{eq:7-0-3}
 &\int_M \Bigl| D_z^2 \Bigl( f_4(z,u_1(\si,z))-f_4(z,u_2(\si,z))\Bigr) \Bigr|^p dz \\
 \le&  C_{p,K} \int_M|(u_1(\si,z)-u_2(\si,z))|^p dz \notag \\
 \le&  C_{p,K} ||( u_1(\cdot,z)-u_2(\cdot,z) )||^p_{C^{\ga_1}([0,T])} dz. \notag  
\end{align}
Thus, 
\begin{align}\label{eq:7-0-2}
 III \le & 
 C_{p,K}\Biggl(E
 \Bigl[
 \int_0^s  d\si
 \int_M \Bigl|  (D_z^{-1})^k D_x^k \bigl(g(t-\si,x,z)-g(s-\si,x,z)\bigr) \Bigl|^2 dz \\
 &\quad \quad \quad\times
 \Bigl(\int_M  || u_1(\cdot,z)-u_2(\cdot,z) ||^p_{C^{\ga_1}([0,T])} dz\Bigr)^{\frac{2}{p}} 
 \Bigr]\Biggr)^{\frac{p}{2}} \notag \\
 \le & 
 C_{p,K}\Biggl(
 \int_0^s  d\si
 \int_M \Bigl|   (D_z^{-1})^k D_x^k \bigl(g(t-\si,x,z)-g(s-\si,x,z)\bigr) \Bigl|^2 dz \notag \\
 &\quad \quad \quad\times
 \Bigl(E \Bigl[\int_M || u_1(\cdot,z)-u_2(\cdot,z) ||^p_{C^{\ga_1}([0,T])} dz 
 \Bigr] \Bigr)^{\frac{2}{p}} \Biggr)^{\frac{p}{2}} \notag 
\end{align}
Furthermore, 
\begin{align}\label{eq:7-0-3}
  \int_0^s d\si\int_M \Bigl|   (D_z^{-1})^2D_x^2 \bigl(g(t-\si,x,z)-g(s-\si,x,z)\bigr) \Bigl|^2 dz\le C_T (t-s)^{\frac{1-\de}{2}},
\end{align}
holds for $\de$ satisfying $\de\in(0,1)$.
To see \eqref{eq:7-0-3}, let us set $g_2(t,x,z)=t^{-\frac12}e^{-K_3 \frac{|x-z|^2}{t}}$, for $0<K_3<2K_2$,
where $K_2$ is the number appearing in \eqref{eq:g-01}.
Then, the left hand side of \eqref{eq:7-0-3} is estimated as
\begin{align*}
   &\int_0^s d\si \int_M \Bigl|  (D_z^{-1})^2D_x^2 \bigl(g(t-\si,x,z)-g(s-\si,x,z)\bigr) \Bigl|^2 dz\\
  =&\int_0^s d\si \int_M \Bigl| \int_{s-\si}^{t-\si} D_v  (D_z^{-1})^2D_x^2 g(v,x,z) dv \Bigr|^2 dz \\
  =&\int_0^s d\si \int_M \Bigl| \int_{\si}^{(t-s)+\si} D_v  (D_z^{-1})^2D_x^2 g(t-\si',x,z) d\si' \Bigr|^2 dz \\
  \le &\int_0^s d\si \int_M \Bigl( \int_{\si}^{(t-s)+\si} 
                            \bigl( D_v  (D_z^{-1})^2D_x^2 g(t-\si',x,z) \bigr)^2 \bigl( \frac{1}{g_2(t-\si',x,z)}\bigr)^{1+\de} d\si' \\
                            &\quad \quad \quad \quad \times \int_{\si}^{(t-s)+\si} g_2(t-\si',x,z)^{1+\de} d\si' \Bigr) dz \\
  \le & C \int_0^s d\si \int_M \Bigl( \int_{\si}^{(t-s)+\si} 
                            \bigl( (t-\si')^{-1} \overline{g}(t-\si',x,z) \bigr)^2 \bigl( \frac{1}{g_2(t-\si',x,z)}\bigr)^{1+\de} d\si' \\
                            &\quad \quad \quad \quad \times \int_{\si}^{(t-s)+\si} (t-\si')^{-\frac{1+\de}{2}} d\si' \Bigr) dz \\
  \le & C \int_0^s d\si \Bigl( \int_{\si}^{(t-s)+\si} (t-\si')^{-2} (t-\si')^{\frac{\de}{2}} d\si'\Bigr)\,(t-s)^{\frac{1-\de}{2}},                       
\end{align*}
for some $C>0$, where at the last inequality, we have used 
\begin{align*}
  \int_{\si}^{(t-s)+\si} (t-\si')^{-\frac{1+\de}{2}} d\si' 
  = \int_{s-\si}^{t-\si} u^{-\frac{1+\de}{2}} du = C((t-\si)^{\frac{1-\de}{2}}-(s-\si)^{\frac{1-\de}{2}})\le C'(t-s)^{\frac{1-\de}{2}},
\end{align*}
since $(s-\si)^{\frac{1-\de}{2}}+(t-s)^{\frac{1-\de}{2}}\ge (t-\si)^{\frac{1-\de}{2}}$ and
\begin{align*}
  \overline{g}(t-\si',x,z)^2 \bigl(\frac{1}{g_2(t-\si',x,z)}\bigr)^{1+\de}
=& \bigl(\frac{K_1}{\sqrt{t-\si'}}\bigr)^2\bigl(\frac{1}{\sqrt{t-\si'}}\bigr)^{-(1+\de)} e^{-(2K_2-K_3)(\frac{x-z}{\sqrt{t-\si'}})^2}\\
\le& C (t-\si')^{\frac{\de}{2}} \frac{1}{\sqrt{t-\si'}}e^{-(2K_2-K_3)(\frac{x-z}{\sqrt{t-\si'}})^2},
\end{align*}
for arbitrary $\de\in(0,1)$.
Thus \eqref{eq:7-0-3} is shown.
Therefore, we get
\begin{align}\label{eq:7-0-4}
 III \le & C_{p.K} (t-s)^{\frac{p}{4}(1-\de)} 
 E \Bigl[\int_M || u_1(\cdot,z)-u_2(\cdot,z) ||^p_{C^{\ga_1}([0,T])} dz \Bigr].
\end{align}
As a result,
\begin{align}\label{eq:7-0-5}
 &E\Bigl[\Bigl|
 D_x^k(w^{(st)}(t,x)-w^{(st)}(s,x)
 \Bigr|^p\Bigr] \\
 \le & C_{p.K} (t-s)^{\frac{p}{4}(1-\de)} 
 E \Bigl[\int_M || u_1(\cdot,z)-u_2(\cdot,z) ||^p_{C^{\ga_1}([0,T])} dz \Bigr],\notag
\end{align}
so that
\begin{align}\label{eq:7-0-5}
 &\int_0^T\int_0^T \frac{E\Bigl[ \int_M \Bigl|
 D_x^k(w^{(st)}(t,x)-w^{(st)}(s,x)
 \Bigr|^p dx \Bigr]}{|t-s|^{1+\a_1 p}} ds dt \\
 \le & C_{p,K} \int_0^T\int_0^T(t-s)^{\frac{p}{4}(1-\de)-(1+\a_1 p)}dsdt 
 E \Bigl[\int_M || u_1(\cdot,z)-u_2(\cdot,z) ||^p_{C^{\ga_1}([0,T])} dz \Bigr]\notag.
\end{align}
%
Since $\a_1\in(0,\frac{1-\de}{4})$, $\a_1 p>1$, $\ga_1\in(0,\a_1-\frac1p)$, it follows
from \eqref{eq:6-0-0-1-0-1} and  \eqref{eq:7-0-5}
\begin{align}\label{eq:7-0-6}
 &E\Bigl[ \int_M || D_x^kw^{(st)}(\cdot,x) ||^p_{W^{\a_1,p}([0,T])} dx \Bigr]\\
 \le & \overline{C}_1
 \bigl( T^{\frac14 p} + T^{\frac{p}{4}(1-\de)-\a_1 p+1} \bigr)
 E \Bigl[\int_M || u_1(\cdot,z)-u_2(\cdot,z) ||^p_{C^{\ga_1}([0,T])} dz \Bigr], \notag
\end{align}
where $\overline{C}_1=\overline{C}_1(C_p, C_{p.K})>0$,
Applying Sobolev embedding's theorem, we get
\begin{align}\label{eq:7-0-7}
 &E\Bigl[ \int_M || D_x^kw^{(st)}(\cdot,x) ||^p_{ W^{\a_1,p}([0,T])} dx \Bigr]\\
 \le & \overline{C}_1 
 \bigl( T^{\frac14 p} + T^{\frac{p}{4}(1-\de)-\a_1 p+1} \bigr)
 E \Bigl[\int_M || u_1(\cdot,z)-u_2(\cdot,z) ||^p_{W^{\a_1,p}([0,T])} dz \Bigr].\notag
\end{align}
The integral of the right hand side can be replaced by 
\begin{align}
 E \Bigl[||u_1-u_2||^p_{\WTX} \Bigr]\notag
\end{align}
and then summing up $k=0,1,2$ in \eqref{eq:6-0-0-1-0-1} and \eqref{eq:7-0-7}, and 
taking $T$ to be sufficiently small, we get

\begin{align}\label{eq:stoch time-001-001-only time}
 \bigl(E\Bigl[ ||w^{(st)}||_{ W^{\a_1,2}_p([0,T]\times M)}^p \Bigr]\bigr)^{\frac1p}
 \le & K_{st}^{(1)} \Bigl(E \Bigl[||u_1-u_2||^p_{\WTX} \Bigr]\Bigr)^{\frac1p},
\end{align}
for $K_{st}^{(1)}=(3\overline{C}_1  \bigl( T^{\frac14 p} + T^{\frac{p}{4}(1-\de)-\a_1 p+1} \bigr))^{\frac1p}$.

%
\subsection{Space regularity for $w^{(st)}$}
In this section, $C_{i,K,p}'$, $i\ge1$ denotes positive constants.
From \eqref{eq:6-0-0-0}, we have
\begin{align}
 E\Bigl[\Bigl|
 D_x^k w^{(st)}(t,x)
 \Bigr|^p\Bigr]
 \le &C_{1,K,p}' 
 t^{( 1-\frac12\frac{p}{p-2} )\frac{p-2}{2}}
 \int_0^t E\Bigl[ || u_1(\si,\cdot)-u_2(\si,\cdot) ||^p_{C^{\ga_2}(M)} \Bigr] d\si,
 \notag
\end{align}
for every $x\in M$. Integrating by $x$, we get
\begin{align}\label{eq:6-0-0-1}
 E
 \Bigl[\Bigl|
 ||D_x^k w^{(st)}(t)||_{L^p(M)}^p\Bigr]
 \le&
 C_{2,K,p}' 
 t^{( 1-\frac12\frac{p}{p-2})\frac{p-2}{2}}
 \int_0^t E\Bigl[ || u_1(\si,\cdot)-u_2(\si,\cdot) ||^p_{C^{\ga_2}(M)} \Bigr] d\si.
\end{align}
Furthermore,
\begin{align}\label{eq:6-0-0}
 &
 E
 \Bigl[\Bigl|
 D_x^k(w^{(st)}(t,x)-w^{(st)}(t,y))
 \Bigr|^p\Bigr]\\
 \le&
 C_p\Bigl(E
 \Bigl[
 \int_0^t \Bigl| 
 \int_M D_x^k\bigl(g(t-\si,x,z)-g(t-\si,y,z)\bigr)\notag\\
 &\times\Bigl( f_4(z,u_1(\si,z))-f_4(z,u_2(\si,z))dz\Bigr|^2 d\si
 \Bigr]\Bigr)^{\frac{p}{2}} \notag\\
 =&
 C_p\Bigl(E
 \Bigl[
 \int_0^t \Bigl| 
 \int_M    (D_z^{-1})^k D_x^k \bigl(g(t-\si,x,z)-g(t-\si,y,z)\bigr)  \notag\\
 &\times  D_z^k \Bigl( f_4(z,u_1(\si,z))-f_4(z,u_2(\si,z))dz\Bigr|^2 d\si
 \Bigr]\Bigr)^{\frac{p}{2}} \notag\\
 \le&
 C_{3,K,p}' \Bigl(E
 \Bigl[
 \int_0^t \Bigl( 
 \int_M \Bigl|   (D_z^{-1})^k D_x^k \bigl(g(t-\si,x,z)-g(t-\si,y,z)\bigr)\Bigr|\notag\\
 &\times | u_1(\si,z)-u_2(\si,z) |dz\Bigr)^2 d\si
 \Bigr]\Bigr)^{\frac{p}{2}},  \notag
\end{align}
holds for every $a\in[0,1]$,
where we have used Burkholder's inequality at the first inequality and 
\eqref{eq:6-0} at the second one.
The last term of \eqref{eq:6-0-0} is bounded from above by
\begin{align}\label{eq:6-1}
 C_{4,K,p}'|x-y|^{ap}\Bigl(E\Bigl[
 \int_0^t (t-\si)^{-\frac12-a} d\si
 \int_M | u_1(\si,z)-u_2(\si,z)|^2dz
 \Bigr]\Bigr)^{\frac{p}{2}}
\end{align}
using Schwarz's inequality and $\int_M \overline{g}(t-\si,y,z)^2dz \le \frac{C}{\sqrt{t-\si}}$.
%
%
%
As a result, using \eqref{eq:6-0-0}, \eqref{eq:6-1}, 
we obtain
\begin{align}\label{eq:6-3-0-0-0}
&\int_M\int_M 
\frac{E
 \Bigl[\Bigl|
 D_x^k(w^{(st)}(t,x)-w^{(st)}(t,y))
 \Bigr|^p\Bigr]}{|x-y|^{1+2\a p}}dxdy\\
\le&
 C_{5,K,p}'
 \int_M\int_M 
 |x-y|^{ap-(1+2\a p)}
 \Bigl(
 E\Bigl[
 \int_0^t (t-\si)^{-\frac12-a} d\si
 \int_M |u_1(\si,z)-u_2(\si,z)|^2dz
 \Bigr]
 \Bigr)^{\frac{p}{2}}
 dxdy
\notag\\
\le& C_{6,K,p}'
 \int_M\int_M |x-y|^{ap-(1+2\a p)} dxdy
 \Bigl(
 E\Bigl[  \int_0^t (t-\si)^{-\frac12-a} d\si
 ||u_1-u_2||_{\WTX}^2
 \Bigr]
 \Bigr)^{\frac{p}{2}}
\notag
\end{align}
Take $p>4$ and $a$ in such a way that $2\a p>1$ and $2\a < a < \frac12$ hold, (here recall $\a\in(0,\frac14)$).
%
Thus, from \eqref{eq:6-0-0-1} and \eqref{eq:6-3-0-0-0},
and then integrating by $t$ at the both hand side, we get
\begin{align}\label{eq:6-10-0}
\Bigl( E \Bigl[ || w^{(st)}||_{W^{0,2+2\a}_p([0,T]\times M)}^p \Bigr]\Bigr)^{\frac{1}{p}} &\le K_{st}^{(2)}
\Bigl(  E\Bigl[ ||u_1-u_2||_{\WTX}^p  \Bigr]\Bigr)^{\frac{1}{p}}.
\end{align}
for $K_{st}^{(2)}=C'(T^{(1-\frac{p}{2(p-2)})\frac{p-2}{2}+1} + T^{\frac{p}{2}(\frac12-2\a)\frac{p}{p-2}+1})^{\frac{1}{p}}>0$,
where $C'=C'(C_{2,K,p}',C_{6,K,p}')$.

%


\subsection{Estimates for $w^{(st)}$}
As a result, it follows from \eqref{eq:stoch time-001-001-only time} and \eqref{eq:6-10-0},
\begin{align}\label{eq:stoch time-001-001}
 \bigl(E\Bigl[ ||w^{(st)}||_{ \WTX}^p \Bigr]\bigr)^{\frac1p}
 \le & K_{st} \Bigl(E \Bigl[||u_1-u_2||^p_{\WTX} \Bigr]\Bigr)^{\frac1p},
\end{align}
for $K_{st}=Const(K_{st}^{(1)},K_{st}^{(2)})>0$.
Taking sufficiently small $T>0$, we can choose $K_{st}$ to be $K_{st}\le \frac12$.
This is possible.
We conclude from \eqref{eq:deter fixed pt} and \eqref{eq:stoch time-001-001}, and Minkowskii's inequality,
\begin{lem}\label{lem:cont-001}
The map $\Ga$ : $Y \mapsto Y$ is contraction;
\begin{align}
  \bigl(E\Bigl[ ||w||_{ \WTX}^p \Bigr]\bigr)^{\frac1p} \le K'  \Bigl(E \Bigl[||u_1-u_2||^p_{\WTX} \Bigr]\Bigr)^{\frac1p},
\end{align}
where  $K'=\min(K_d,K_{st})$ and $w=\Ga(u_1)-\Ga(u_2)$ as defined in Section 4.1.  
\end{lem}

\subsection{The property $\Ga(Y)\subset Y$}
We have already $\Ga$ is a contraction map from $Y$ onto $Y$.
We will show $\Ga(Y)\subset Y$ in this section.
For $u\in Y$, 
\begin{align*}
  |||\Ga u||| =& |||\Ga u - \Ga 0 + \Ga 0 ||| \le |||\Ga u - \Ga 0|||+ ||| \Ga 0 ||| \le K' |||u||| + ||| \Ga 0 |||  \le \frac12R +||| \Ga 0 |||,
\end{align*}
holds, where $|||\cdot|||$ is the norm defined by \eqref{eq:distance of X}. Therefore, it suffices to show 
\begin{align}
||| \Ga 0 ||| \le \frac12 R.
\end{align}
Note that $\Ga0$ is written as
\begin{align}
 (\Ga0)(t,x)=&\int_0^t \int_M g(t-\si,x,z)\tilde{F}(z,0,0,0)dzd\si\\
 &+\int_0^t \int_M g(t-\si,x,z)f_4(z,0)dzdB_\si \equiv  \,w^{(d,0)}(t,x)+w^{(st,0)}(t,x). \notag
\end{align}
Then we have
\begin{lem}\label{lem:w(0,d)}
\begin{align}\label{eq:w(0,d)}
||w^{(d,0)}||_{\WTX} \le C_3 ||\tilde{F}||_{W^{2\a,p}(M)}T^{\frac1p}.
\end{align}
holds.
\end{lem}
\begin{proof}
Since $w^{(d,0)}$ is a solution of a linear equation \eqref{eq:sol of (d) - 001} with $\psi(t,x)$ replaced by 
$\tilde{F}(x,0,0,0)$,
\begin{align}\label{eq:w(0,d)}
||w^{(d,0)}||_{\WTX} \le C_3 ||\tilde{F}||_{W^{\a_1,2\a}_p([0,T]\times M)},
\end{align}
follows from  \eqref{eq:Schauder of sol of (d) - 000-000-001}.
The assertion is obtained from the above inequality.
\end{proof}
In addition, the following lemma is obtained:
\begin{lem}\label{lem:w(0,st)}
\begin{align}\label{eq:w(0,st)}
\bigl(E\Bigl[||w^{(st,0)}||_{\WTX}^p\Bigr]\bigr)^{\frac1p} 
\le \Bigl( ({K_{st}^{(1)}})^p + ({K_{st}^{(2)}})^p \Bigr)^{\frac1p} \sup_{i=0,1,2}\sup_{z\in M}| D_z^i f_4(z,0)|.
\end{align}
\end{lem}
\begin{proof}
Note that $w^{(st,0)}$ has a form of \eqref{eq:the stochastic integral w(st)} with 
$f_4(z,u_1(\si,z))-f_4(z,u_2(\si,z))$ replaced by $f_4(z,0)$.
Computing similarly as in \eqref{eq:6-0-0-0}, however in this case by noting that $\sup_{i=0,1,2}\sup_{z\in M}| D_z^i f_4(z,0)|$ is finite, 
we get 
\begin{align}\label{eq:w(0,st)-001}
 E\Bigl[\Bigl|
 D_x^k w^{(st,0)}(t,x)
 \Bigr|^p\Bigr]
 \le &
 C_{p,K} t^{(1-\frac{p}{2(p-2)})\frac{p-2}{2}+1} \bigl(\sup_{i=0,1,2}\sup_{z\in M}| D_z^i f_4(z,0)|\bigr)^p,
\end{align}
for $k=0,1,2$.
Similarly, 
$w^{(st,0)}(t,x)-w^{(st,0)}(s,x)$ has a form of \eqref{eq:diff of w(st)} with 
$f_4(z,u_1(\si,z))-f_4(z,u_2(\si,z))$ replaced by $f_4(z,0)$.
Furthermore, 
$w^{(st,0)}(t,x)-w^{(st,0)}(t,y)$ has a form of \eqref{eq:6-0-0} with 
$f_4(z,u_1(\si,z))-f_4(z,u_2(\si,z))$ replaced by $f_4(z,0)$.

Noting again $\sup_{i=0,1,2}\sup_{z\in M}| D_z^i f_4(z,0)|$ is finite, by similar method to obtain \eqref{eq:7-0-6}, we get
\begin{align}\label{eq:w(0,st)-002}
 &E\Bigl[ \int_M || D_x^kw^{(st,0)}(\cdot,x) ||^p_{W^{\a_1,p}([0,T])} dx \Bigr]\\
 \le & \overline{C}_1 
 \bigl( T^{\frac{p}{4}} + T^{\frac{p}{4}(1-\de)-\a_1 p+1} \bigr) \bigl(\sup_{i=0,1,2}\sup_{z\in M}| D_z^i f_4(z,0)|\bigr)^p, \notag
\end{align}
for $k=0,1,2$.
Similarly, the following estimate for each $k=0,1,2$ is obtained:
\begin{align}
 &E\Bigl[ \int_0^T || D_x^kw^{(st,0)}(t,\cdot) ||^p_{W^{2\a,p}(M)} dt \Bigr]\\
 \le &  
 C''(T^{(1-\frac{p}{2(p-2)})\frac{p-2}{2}+1} + T^{\frac{p}{2}(\frac12-2\a)\frac{p}{p-2}+1})\bigl(\sup_{i=0,1,2}\sup_{z\in M}| D_z^i f_4(z,0)|\bigr)^p, \notag
\end{align}
for some constant $C''>0$. 
The assertion follows from \eqref{eq:w(0,st)-001} and \eqref{eq:w(0,st)-002}.
\end{proof}

From those lemmas, we have
\begin{align}
&E\Bigl[|| \Ga 0 ||^p_{\WTX} \Bigr]\\
\le&
2^{p}\Bigl( E\Bigl[|| w^{(d,0)}||^p_{\WTX} \Bigr] + E\Bigl[|| w^{(st,0)}||^p_{\WTX} \Bigr]\Bigr) \notag\\
\le& 2^{p} \Bigl( C_3^p ||\tilde{F}||_{W^{2\a,p}(M)}^pT 
+
( ({K_{st}^{(1)}})^p + ({K_{st}^{(2)}})^p )
\bigl(\sup_{i=0,1,2}\sup_{z\in M}| D_z^i f_4(z,0)|\bigr)^p \Bigr)\notag
\end{align}
As a result, for each $R>0$, choose $T>0$ satisfying
\begin{align}\label{eq:the way of choosing T}
 2^{p} \Bigl( C_3^p ||\tilde{F}||_{W^{2\a,p}(M)}^pT 
+
( ({K_{st}^{(1)}})^p + ({K_{st}^{(2)}})^p )
\bigl(\sup_{i=0,1,2}\sup_{z\in M}| D_z^i f_4(z,0)|\bigr)^p \Bigr) \le \frac{R}{2}.
\end{align}
Then, we obtain
\begin{lem}\label{lem:ga is onto map}
Under the condition where $T>0$ fulfils \eqref{eq:the way of choosing T}, $\Ga(Y)\subset Y$ holds. 
\end{lem}

\subsection{Proof of the main theorem}
\begin{proof}
Let us take a sufficiently small $T>0$ in such a way that \eqref{eq:the way of choosing T} holds.
Applying Lemma \ref{lem:cont-001} and Lemma \ref{lem:ga is onto map} for $w=w^{(d)}+w^{(st)}$,
the assertion is obtained.
In particular, $\si_K>0$, a.s., which follows from the fact that $D^iu$ ($i=0,1,2$) are ${\ga_1}$-continuous in $t$, a.s., using Sobolev embedding theorem (see Section 4.1)
and $u(0,x)=0$ for every $x\in M$.
The proof is complete. 
\end{proof}

\section*{Appendix}
\subsection{Regularity for the fundamental solution}

Friedman \cite{FR64} shows that
the fundamental solution $g(t,x,y)$, $x,y\in\R$ of the second order parabolic linear differential operator $\frac{\partial}{\partial t}-A$ has
the following estimate:
\begin{align*}
\Bigl|D_t^j D_x^\a D_y^\b g(t,x,y)\Bigr|\le t^{-\frac{\a+\b}{2}-j}\overline{g}(t,x,y),
\end{align*}
where 
\begin{align}\label{eq:g-01}
\overline{g}(t,x,y)=K_1t^{-\frac{1}{2}}\exp\Bigl( {-K_2\frac{|x-y|^{2}}{t}\Bigr)},
\end{align}
and $K_1$, $K_2>0$ are constants depending on $\a$, $\b$, $j \in \Z_+$.
In our case, we make use of the estimate in \cite{FR64}, (see also Funaki \cite{F91a} and Yokoyama \cite{Y}). 
%
\begin{lem}
\begin{align}
\Bigl|D_y^{-k} D_x^\a  g(t,x,y)\Bigr|=\le t^{-\frac{\a}{2}+\frac{k}{2}}\overline{g}(t,x,y),\quad t>0,\,x,y\in M,
\end{align}
holds for $k,\a\in \{0,1,2\}$ with $k\le\a$ and suitable $K_1$ and $K_2>0$.
\end{lem}
\begin{proof}
This is obtained by the estimates for the Gaussian kernel hence we only sketch the outline of the proof.
Take $z_0 < z$ sufficiently small satisfying $|z-x|\le|z_0-x|$ for every $x,z\in M$. 
Then we get 
\begin{align}\label{eq:ap-001}
\Bigl|D_x^2 \int_{z_0}^z g(t,x,\tilde{z}) d\tilde{z}\Bigr|
\le t^{-1} \int_{z_0}^z \overline{g}(t,x,\tilde{z}) d\tilde{z}
\le K_1' t^{-\frac12} t^{-\frac12} e^{-{K_2'}\frac{|z-x|^2}{t}}, 
\end{align}
for some $K_1'>0$ and $K_2'\in(0,K_2)$ and by changing of variables. 
Proceeding similar argument again, the estimates in the case of $k=2$ is also obtained.
\end{proof}

\section*{Acknowledgements}
This research is supported by JSPS KIKIN Grant 18K13430.
The author would like to thank to Professor Martina Hofmanov\`a, Professor Tadahisa Funaki and Pierre Simonot for helpful discussion.
The author also would like to express his thanks to Professor Panagiotis Souganidis for a comment to the references.

\end{document}